\documentclass{amsart}
\usepackage{amsthm,amssymb,amsmath}
\usepackage{epsfig}
\usepackage{color}

\newtheorem{prelem}{{\bf Proposition}}

\newtheorem{theorem}{Theorem}
\newtheorem{corollary}[theorem]{Corollary}

\newtheorem{observation}[theorem]{Observation}
\newtheorem{proposition}[theorem]{Proposition}

\theoremstyle{definition}

\theoremstyle{remark}

\input{tcilatex}

\begin{document}
\title{$k$-tuple total domination in inflated graphs}
\author{ Adel P. Kazemi \vspace{4mm}\\
Department of Mathematics\\
University of Mohaghegh Ardabili\\
P. O. Box 5619911367, Ardabil, Iran\\
adelpkazemi@yahoo.com\vspace{3mm} \\
}
\date{}
\maketitle

\begin{abstract}
The inflated graph $G_{I}$ of a graph $G$ with $n(G)$ vertices is
obtained from $G$ by replacing every vertex of degree $d$ of $G$ by
a clique, which is isomorph to the complete graph $K_{d}$, and each
edge $(x_{i},x_{j})$ of $G$ is replaced by an edge $(u,v)$ in such a
way that $u\in X_{i}$, $v\in X_{j}$, and two different edges of $G$
are replaced by non-adjacent edges of $G_{I}$. For integer $k\geq
1$, the $k$-tuple total domination number $\gamma _{\times k,t}(G)$
of $G$ is the minimum cardinality of a $k$-tuple total dominating
set of $G$, which is a set of vertices in $G$ such that every vertex
of $G$ is adjacent to at least $k$ vertices in it. For existing this
number, must the minimum degree of $G$ is at least $k$. Here, we
study the $k$-tuple total domination number in inflated graphs when
$k\geq 2$. First we prove that $n(G)k\leq \gamma _{\times
k,t}(G_{I})\leq n(G)(k+1)-1$, and then we characterize graphs $G$
that the $k$-tuple total domination number number of $G_I$ is
$n(G)k$ or $n(G)k+1$. Then we find bounds for this number in the
inflated graph $G_I$, when $G$ has a cut-edge $e$ or cut-vertex $v$,
in terms on the $k$-tuple total domination number of the inflated
graphs of the components of $G-e$ or $v$-components of $G-v$,
respectively. Finally, we calculate this number in the inflated
graphs that have obtained by some of the known graphs.
\end{abstract}


\textbf{Keywords :} $k$-tuple total domination number, inflated graph

\textbf{2000 Mathematics subject classification :} 05C69

\section{Introduction}

All graphs considered here are finite, undirected, and simple. For standard
graph theory terminology not given here we refer to \cite{hhs1}. Let $%
G=(V,E) $ be a graph with \emph{vertex set} $V$ of $order$ $n(G)$
and \emph{edge set} $E$ of $size$ $m(G)$. The \emph{open
neighborhood} of a vertex $v\in V$ is $N_{G}(v)=\{u\in V\ |\ uv\in
E\}$ and its \emph{closed neighborhood} is $N_{G}[v]=N_{G}(v)\cup
\{v\}$. The \emph{degree} of a vertex $v$ is also $deg_G(v)=\mid
N_{G}(v) \mid $. The \emph{minimum} and \emph{maximum degree} of $G$
are respectively denoted by $\delta =\delta (G)$ and $\Delta =\Delta
(G)$. We say that a graph is $connected$ if there exist a path
between every two vertices of the graph, and otherwise is called
$disconnected$. In a connected graph $G$, a vertex (resp. edge) $v$
is called a $cut$-$vertex$ or (resp. $cut$-$edge$) if $G-v$ is
disconnected.
Every maximal connected subgraph of $G-v$ is called a ($connectedness$) $%
component$ of it. Let $v$ be a cut-vertex of a graph $G$ and $S$ be the
vertex set of a component of $G-v$. The induced subgraph by $S\cup \{v\}$ of
$G$ we call a $v$-$component$ of $G$.

An edge subset $M$ in $G$ is called a $matching$ in $G$ if any two
edges of $M$ has no vertex in common. If $e=vw\in M$, then we say either $M$ $%
saturate$ two vertices $v$ and $w$ or $v$ and $w$ are $M$-$saturated$ (by $e$%
). A matching M is a $perfect$ $matching$ if all vertices of $G$ are
$M$-saturated. Also a matching $M$ is a $maximum$ $matching$ if
there is no other matching $M^{\prime }$ with $\mid M^{\prime }\mid
>\mid M\mid $. In a graph $G$ the number of edges in a maximum
matching is denoted by $\alpha ^{\prime }(G)$.

Domination in graphs is now well studied in graph theory and the
literature on this subject has been surveyed and detailed in the two
books by Haynes, Hedetniemi, and Slater~\cite{hhs1, hhs2}. A set
$S\subseteq V$ is a \emph{total dominating set} if each vertex in
$V$ is adjacent to at least one vertex of $S$, while the minimum
cardinality of a total dominating set is the \emph{total domination
number} $\gamma _{t}(G)$ of $G$.

In \cite{HeKa09} Henning and Kazemi generalized this definition to
the $k$-tuple total domination number as follows: a subset $S$ of
$V$ is a \emph{$k$-tuple total dominating set} of $G$, abbreviated
kTDS, if for every vertex $v\in V $, $\mid N(v)\cap S\mid \geq k$;
that is, $S$ is a kTDS if every vertex has at
least $k$ neighbors in $S$. The \emph{$k$-tuple total domination number} $%
\gamma _{\times k,t}(G)$ of $G$ is the minimum cardinality of a kTDS
of $G$. We remark that $\gamma _{t}(G)=\gamma _{\times 1,t}(G)$. For
a graph to have a $k$-tuple total dominating set, its minimum degree
is at least $k$. Since every (k+1)TDS is also a kTDS, we note that
$\gamma _{\times k,t}(G)\leq \gamma _{\times (k+1),t}(G)$ for all
graphs with minimum degree at least $k+1$. A kTDS of cardinality
$\gamma _{\times k,t}(G)$ we call a $\gamma _{\times
k,t}(G)$-set. When $k=2$, a $k$-tuple total dominating set is called a \emph{%
double total dominating set}, abbreviated DTDS, and the $k$-tuple total
domination number is called the \emph{double total domination number}. The
redundancy involved in $k$-tuple total domination makes it useful in many
applications.

For the notation for inflated graphs, we follow that of \cite{Fa98}. The $%
infation$ or $infated$ graph $G_{I}$ of the graph $G$ without
isolated vertices is obtained as follows: each vertex $x_{i}$ of
degree $d(x_{i})$ of $G$ is replaced by a clique $X_{i}\cong
K_{d(x_{i})}$ (that is, $X_{i}$
is isomorph to the complete graph $K_{d(x_{i})}$) and each edge $%
(x_{i},x_{j})$ of $G$ is replaced by an edge $(u,v)$ in such a way that $%
u\in X_{i}$, $v\in X_{j}$, and two different edges of $G$ are
replaced by non-adjacent edges of $G_{I}$. An obvious consequence of
the definition is that $n(G_{I})=\sum_{x_{i}\in
V(G)}d_{G}(x_{i})=2m(G)$, $\delta (G_{I})=\delta (G)$ and $\Delta
(G_{I})=\Delta (G)$. There are two different kinds of edges in
$G_{I}$. The edges of the clique $X_{i}$ are colored red and the
$X_{i}$'s are called the $red$ $cliques$ (a red clique $X_{i}$ is
reduced to a point if $x_{i}$ is a pendant vertex of $G$). The other
ones, which correspond to the edges of $G$, are colored $blue$ and
they form a perfect matching of $G_{I}$. Every vertex of $G_{I}$
belongs to exactly one
red clique and one blue edge. Two adjacent vertices of $G_{I}$ are said to $%
red$-$adjacent$ if they belong to a same red clique, $blue$-$adjacent$
otherwise. In general, we adopt the following notation: if $x_{i}$ and $%
x_{j} $ are two adjacent vertices of $G$, the end vertices of the blue edge
of $G_{I}$ replacing the edge $(x_{i},x_{j})$ of $G$ are called $x_{i}x_{j}$
in $X_{i}$ and $x_{j}x_{i}$ in $X_{j}$, and this blue edge is $%
(x_{i}x_{j},x_{j}x_{i})$. Clearly an inflation is claw-free. More precisely,
$G_{I}$ is the line-graph $L(S(G))$ where the subdivision $S(G)$ of $G$ is
obtained by replacing each edge of $G$ by a path of length 2. The study of
various domination parameters in inflated graphs was originated by Dunbar
and Haynes in \cite{DH96}. Results related to the domination parameters in
inflated graphs can be found in \cite{Fa98,Fa01,Pu00}.

Henning and Kazemi in \cite{HK} discussed on total domination number in
inflated graphs which is the same $k$-tuple total domination number when $%
k=1 $. Here we continue the studying of the $k$-tuple total
domination number in inflated graphs when $k\geq 2$. This paper is
organized as follows. In section 2, we prove that if $k\geq 2$ is an
integer and $G$ is a graph of order $n$ with $\delta \geq k$%
, then $nk\leq \gamma _{\times k,t}(G_{I})\leq n(k+1)-1$, and then
we characterize graphs $G$ that $\gamma _{\times k,t}(G_{I})$ is
$nk$ or $nk+1$. In section 3, we find upper and lower bounds for the
$k$-tuple total domination number of the inflation of a graph $G$,
which contains a cut-edge $e$, in terms on the $k$-tuple total
domination number of the inflation of the components of $G-e$. Also
in a similar manner, we find upper and lower bounds for the
$k$-tuple total domination number of the inflation of a graph $G$,
which contains a cut-vertex $v$, in terms on the $k$-tuple total
domination number of the inflation of the $v$-components of $G-v$.
Also we find the $k$-tuple total domination number of the inflation
of the complete graphs. Finally, in section 4, we calculate the
$k$-tuple total domination number in the inflation of the known
graphs: the generalized Petersen graphs, Harary graphs and complete
bipartite graphs. Also we give an upper bound for this number in the
inflation of the complete multipartite graphs.

\section{general bbounds}

First we give two general upper and lower bounds for the $k$-tuple
total domination number of inflated graphs, where $\delta \geq k\geq
2$.

\begin{theorem}
\label{LU.bounds} Let $k\geq 2$ be an integer, and let $G$ be a graph of order $n$ with $%
\delta \geq k$. Then
\begin{equation*}
nk\leq \gamma _{\times k,t}(G_{I})\leq n(k+1)-1.
\end{equation*}
\end{theorem}

\begin{proof}
Let $V(G)=\{x_{i}\mid 1\leq i\leq n\}$ and let $S$ be an arbitrary
kTDS of $G_{I}$. Since every vertex of the red clique $X_{i}$\ is
adjacent to only one vertex of another red clique, then $\mid S\cap
X_{i}\mid \geq \mid N_{X_{i}}[v]\mid \geq k$, for each vertex $v\in
S\cap X_{i}$ and hence $\gamma _{\times k,t}(G_{I})\geq nk$.

Now we prove $\gamma _{\times k,t}(G_{I})\leq n(k+1)-1$. Set
$S_{1}=\{x_{1}x_{j}\mid 2\leq j\leq k+1\}$ as a subset of $X_{1}$.
For each $2\leq j\leq n$, let $S_{j}$ be a ($k+1$)-subset of $X_{j}$
such that $x_{j}x_{1}\in S_{j}$, for each $2\leq j\leq k+1$. Since
$S_{1}\cup
S_{2}\cup ...\cup S_{n}$ is a kTDS of $G_I$ with cardinal $n(k+1)-1$, then $%
\gamma _{\times k,t}(G_{I})\leq n(k+1)-1$.
\end{proof}

We recall the next proposition from \cite{HeKa09}.

\begin{prelem}
\label{HK.Lbound} \emph{(Kazemi, Henning \cite{HeKa09} 2010)} Let
$G$ be a graph with minimum degree at least $k$. If $k\geq 2$ is an
integer, then
\begin{equation*}
\gamma_{\times k,t}(G)\geq \lceil \frac{kn}{\Delta (G)}\rceil.
\end{equation*}
\end{prelem}

By Proposition \ref{HK.Lbound} and Theorem \ref{LU.bounds} we have
the next result.

\begin{corollary}
\label{L.bound,max} If $G$ is a graph of order $n$ and size $m$ with
$\delta (G)\geq k\geq 2$, then
\begin{equation*}
\gamma _{\times k,t}(G_{I})\geq \max \{nk,\lceil 2km/\Delta (G)
\rceil \}.
\end{equation*}
\end{corollary}

Let $k=\delta (G)$. Then, since every red clique of cardinal $k$ is
subset of every kTDS of $G_{I}$, Theorem \ref{LU.bounds} can be
improved in such a way.

\begin{corollary}
\label{LU.bounds,delta} Let $G$ be a graph of order $n$ with $\delta
\geq 2$. If $\ell$ is the number of vertices in $G$ of degree
$\delta $, then $n\delta \leq \gamma _{\times \delta ,t}(G_{I})\leq
n(\delta +1)-\ell$.
\end{corollary}

Now, we characterize graphs $G$ of order $n$ that the $k$-tuple
total domination number of their inflation is $nk$ or $nk+1$. First
we give the next two new definitions.

\textbf{Two new definitions:}

We know that a graph $G$ is a \emph{Hamiltonian graph} if it has a
\emph{Hamiltonian cycle}, that is, a cycle that contains all
vertices of the graph. We extend this definition in such a way: a
graph $G$ is a \emph{Hamiltonian-like decomposable graph} if there
are disjoint Hamiltonian subgraphs $G_{1}$, $G_{2}$, ..., $G_{t}$ of
$G$ such that $V(G)=V(G_{1})\cup V(G_{2})\cup ...\cup V(G_{t})$. A
such partition we call a \textit{Hamiltonian-like decomposition} of
$G$ and simply we write $G=HLD(G_{1},G_{2},...,G_{t})$. In
generally, for each integer $k\geq 1$, we say that a graph $G$ is a
$k$-\emph{Hamiltonian-like decomposable graph}, briefly kHLD-graph,
if it has $k$ Hamiltonian-like decomposition
$G=HLD(G_{1}^{(i)},G_{2}^{(i)},...,G_{t_{i}}^{(i)})$ of Hamiltonian
subgraphs (where $1\leq i\leq k$) such that for every two distinct
Hamiltonian subgraphs $G_{s_{i}}^{(i)}$ and $G_{s_{j}}^{(j)}$, their
Hamiltonian cycles $C_{s_{i}}^{(i)}$ and $C_{s_{j}}^{(j)}$ are
disjoint. We note that $1$-Hamiltonian-like decomposable graph is
the same Hamiltonian-like decomposable graph.

A $k$-Hamiltonian-like decomposable graph $G$, we call kHLPM-graph
or kHLMM-graph if $G$ has respectively a perfect or maximum matching
$M$ with cardinal $\lfloor n/2\rfloor $ such that for each partition
$G=HLD(G_{1}^{(i)},G_{2}^{(i)},...,G_{t_{i}}^{(i)})$ of Hamiltonian
subgraphs (where $1\leq i\leq k$), $M$ satisfies in the following
condition:
\begin{equation}
\begin{array}{lll}
M\cap E(C_{\ell_{i}}^{(i)})=\emptyset ,\mbox{ for each }%
1\leq \ell_{i}\leq t_{i},
\end{array}  \label{eqq}
\end{equation}
where $C_{\ell_{i}}^{(i)}$ is the Hamiltonian cycle of
$G_{\ell_{i}}^{(i)}$.

The next two theorems characterize graphs $G$ with $\gamma _{\times
k,t}(G_{I})=nk$.

\begin{theorem}
\label{gamma=2kn} Let $G$ be a graph of order $n$ and let $1\leq 2k\leq \delta $. Then $%
\gamma _{\times (2k),t}(G_{I})=2kn$ if and only if $G$ is a
kHLD-graph.
\end{theorem}

\begin{proof}
Let $V(G)=\{x_{i}\mid 1\leq i\leq n\}$. For each $1\leq i\leq k$ and some $t_{i}\geq 1$, let $%
G=HLD(G_{1}^{(i)},G_{2}^{(i)},...,G_{t_{i}}^{(i)})$ be a
Hamiltonian-like decomposition of $G$. For each $1\leq i\leq k$ and
each $1\leq \ell_{i}\leq t_{i}$, let
$C_{\ell_{i}}^{(i)}:x_{1}^{(i)}x_{2}^{(i)}...x_{c_{i,\ell_{i}}}^{(i)}$\
be the Hamiltonian cycle of $G_{\ell_{i}}^{(i)}$. Set
\begin{equation*}
S_{i,\ell_{i}}=\{x_{m}^{(i)}x_{m-1}^{(i)},x_{m}^{(i)}x_{m+1}^{(i)}\mid
1\leq m\leq c_{i,\ell_{i}}\}.
\end{equation*}
Then every $S^{(i)}=S_{i,1}\cup S_{i,2}\cup ...\cup S_{i,t_{i}}$ is
a DTDS of $G_{I}$ with cardinal $2n$. Since $G$ is
$k$-Hamiltonian-like decomposable, then every two distinct
$S^{(i)}$\ and $S^{(\ell)}$ are disjoint and hence
$S^{(1)}\cup S^{(2)}\cup ...\cup S^{(k)}$ is a 2kTDS of $G_{I}$%
 with cardinal $2kn$. Thus $\gamma _{\times (2k),t}(G_{I})\leq 2kn$
and Theorem \ref{LU.bounds} follows $\gamma _{\times
(2k),t}(G_{I})=2kn$.

Conversely, let $\gamma _{\times (2k),t}(G_{I})=2kn$ and let $S$ be a $%
\gamma _{\times (2k),t}(G_{I})$-set. Since for each $1\leq i\leq n$,
$\mid
S\cap X_{i}\mid =2k$, then we may partition every $S\cap X_{i}$\ to $k$ $2$%
-subsets $D_{j}^{(i)}$, where $1\leq j\leq k$, such that
$D_{j}^{(1)}\cup D_{j}^{(2)}\cup ...\cup D_{j}^{(n)}$ is a union of
some disjoint cycles. Without loss of generality, we may assume that
$D_{j}^{(1)}\cup D_{j}^{(2)}\cup ...\cup D_{j}^{(n)}$\ is the cycle
\begin{equation*}
C_{j}:x_{1}x_{n},x_{1}x_{2};x_{2}x_{1},x_{2}x_{3};x_{3}x_{2},x_{3}x_{4};...;x_{n}x_{n-1},x_{n}x_{1}.
\end{equation*}
Then $G$ has the corresponding cycle $C_{j}^{\prime
}:x_{1}x_{2}x_{3}x_{4}...x_{n}$. Thus for every partition
$D_{j}^{(1)}\cup
D_{j}^{(2)}\cup ...\cup D_{j}^{(n)}$ there is a corresponding partition $%
G=HLD(G_{1}^{(i)},G_{2}^{(i)},...,G_{t_{i}}^{(i)})$\ of Hamiltonian
subgraphs $G_{1}^{(i)}$, $G_{2}^{(i)}$, ... and $G_{t_{i}}^{(i)}$,
and so $G$ is a kHLD-graph.
\end{proof}

\begin{theorem}
\label{gamma=(2k+1)n} Let $G$ be a graph of order $n$ and let $1\leq
2k+1\leq \delta $. Then $\gamma _{\times (2k+1),t}(G_{I})=(2k+1)n$
if and only if $G$ is a kHLPM-graph.
\end{theorem}

\begin{proof}
Let $V(G)=\{x_{i}\mid 1\leq i\leq n\}$. Let $G$ be a kHLPM-graph. We
follow exactly the notation and terminology introduced in the first
and second paragraphs of the proof of Theorem \ref{gamma=2kn}. Then
similarly $S^{(1)}\cup S^{(2)}\cup ...\cup S^{(k)}$
is a 2kTDS of $G_{I}$\ with cardinal $2kn$. Set $M_{I}=%
\{(x_{i}x_{j},x_{j}x_{i})\mid x_{i}x_{j}\in M\}$. Since for every partition $%
G=HLD(G_{1}^{(i)},G_{2}^{(i)},...,G_{t_{i}}^{(i)})$\ of Hamiltonian
subgraphs, $M$ satisfies in the condition (\ref{eqq}), then
$V(M_{I})\cap
(S^{(1)}\cup S^{(2)}\cup ...\cup S^{(k)})=\emptyset $. One can verify that $%
V(M_{I})\cup S^{(1)}\cup S^{(2)}\cup ...\cup S^{(k)}$ is a (2k+1)TDS of $%
G_{I}$ with cardinal $(2k+1)n$. Thus $\gamma _{\times
(2k+1),t}(G_{I})\leq (2k+1)n$ and Theorem \ref{LU.bounds} follows
$\gamma _{\times (2k+1),t}(G_{I})=(2k+1)n$.

Conversely, let $\gamma _{\times (2k+1),t}(G_{I})=(2k+1)n$ and let
$S$ be a $\gamma _{\times (2k+1),t}(G_{I})$-set. Since for each
$1\leq i\leq n$, $\mid S\cap X_{i}\mid =2k+1$, then, similar to the
proof of Theorem \ref{gamma=2kn}, we
may partition every $S\cap X_{i}$\ to $k$ $2$-subsets $D_{j}^{(i)}$, where $%
1\leq j\leq k$, such that $D_{j}^{(1)}\cup D_{j}^{(2)}\cup ...\cup
D_{j}^{(n)}$ is a union of some disjoint cycles and there is a
corresponding partition
$G=HLD(G_{1}^{(i)},G_{2}^{(i)},...,G_{t_{i}}^{(i)})$ of Hamiltonian
subgraphs for it, and also $\cup _{1\leq i\leq n}(S-(\cup _{1\leq
j\leq k}D_{j}^{(i)}))$ makes a blue matching $M_{I}$\ in $G_{I}$ of
size $\lfloor n/2\rfloor $. It can be easily verified that
$M=\{x_{i}x_{j}\mid (x_{i}x_{j},x_{j}x_{i})\in M_{I}\}$ is a perfect
matching in $G$ that satisfies in the condition (\ref{eqq}), and so
$G$ is a kHLPM-graph.
\end{proof}

Theorems \ref{LU.bounds}, \ref{gamma=2kn} and \ref{gamma=(2k+1)n}
follow the next result.

\begin{theorem}
\label{gamma>=nk+1} Let $G$ be a graph of order $n$, and let $1\leq
k\leq \delta $. Then
\[
nk+1 \leq \gamma _{\times k,t}(G_{I})\leq n(k+1)-1
\]
if and only if either $k$ and $n$ are both odd or if $k$ is even or
odd, then respectively $G$ is not a kHLD- or kHLPM-graph.
\end{theorem}

By closer look at the proofs of Theorems \ref{gamma=2kn} and
\ref{gamma=(2k+1)n} we have the following observation.

\begin{observation}
\label{obser} Let $k$ be an integer and let $G$ be a graph of order
$n$ with $\gamma _{\times k,t}(G)=nk$. Then for every $\gamma
_{\times k,t}(G_{I})$-set $S$, the induced subgraph $G_{I}[S]$ of
$S$ in $G_{I}$ contains a union of disjoint Hamiltonian cycles (of
some of the its subgraphs) and probably a perfect matching.
Therefore, if we reduce the number of vertices of $S$ in a red
clique of $G_{I}$ to less than $k$ vertices, then there is another
unique red clique $X$ of $G_{I}$\ and an unique vertex $w$\ of
$X\cap S$ such that $w$ is not $k$-tuple totally dominated by $S$.
\end{observation}

The next theorem states an equivalent condition for $\gamma _{\times
k,t}(G_{I})=nk+1$, when $k$ and $n$ are both odd.

\begin{theorem}
\label{gamma=n(2k+1)+1} Let $G$ be a graph of odd order $n$ and let $1\leq 2k+1\leq \delta $. Then $%
\gamma _{\times (2k+1),t}(G_{I})=(2k+1)n+1$ if and only if $G$ is a
kHLMM-graph.
\end{theorem}

\begin{proof}
Let $V(G)=\{x_{i}\mid 1\leq i\leq n\}$. Let $G$ be a kHLMM-graph.
Without loss of generality, we may assume that $M$ does not saturate
$x_{n}$. For each $1\leq i\leq k$ and $t_{i}\geq 1$, let
$G=HLD(G_{1}^{(i)},G_{2}^{(i)},...,G_{t_{i}}^{(i)})$ be a
Hamiltonian-like decomposition. For each $1\leq \ell_{i}\leq
t_{i}$, let $C_{\ell_{i}}^{(i)}:x_{1}^{(i)}x_{2}^{(i)}...x_{c_{i,\ell_{i}}}^{(i)}$%
\ be a Hamiltonian cycle for $G_{\ell_{i}}^{(i)}$. Set
\begin{equation*}
S_{i,\ell_{i}}=\{x_{m}^{(i)}x_{m-1}^{(i)},x_{m}^{(i)}x_{m+1}^{(i)}\mid
1\leq m\leq c_{i,\ell_{i}}\}.
\end{equation*}
Then every $S^{(i)}=S_{i,1}\cup S_{i,2}\cup ...\cup S_{i,t_{i}}$ is
a DTDS of $G_{I}$ with cardinal $2n$. Since $G$ is
$k$-Hamiltonian-like decomposable, then every two distinct $S^{(i)}$
and $S^{(\ell)}$ are disjoint and hence $S^{(1)}\cup S^{(2)}\cup
...\cup S^{(k)}$ is a 2kTDS of $G_{I}$ with cardinal $2kn$. Set
$M_{I}=\{(x_{i}x_{j},x_{j}x_{i})\mid x_{i}x_{j}\in M\}$. Since for
each partition $G=HLD(G_{1}^{(i)},G_{2}^{(i)},...,G_{t_{i}}^{(i)})$
of Hamiltonian subgraphs, $M$ satisfies in the condition
(\ref{eqq}), then $V(M_{I})\cap (S^{(1)}\cup S^{(2)}\cup ...\cup
S^{(k)})=\emptyset $. One can verify that for every two arbitrary
vertices $\alpha ,\beta \in X_{n}-(S^{(1)}\cup S^{(2)}\cup ...\cup
S^{(k)})$, the set $V(M_{I})\cup S^{(1)}\cup S^{(2)}\cup ...\cup
S^{(k)}\cup \{\alpha ,\beta \}$ is a $(2k+1)$TDS of $G_{I}$ with
cardinal $(2k+1)n+1$. Thus $\gamma _{\times (2k+1),t}(G_{I})\leq
(2k+1)n+1$ and Theorem \ref{gamma>=nk+1} follows $\gamma _{\times
(2k+1),t}(G_{I})=(2k+1)n+1$.

Conversely, let $\gamma _{\times (2k+1),t}(G_{I})=(2k+1)n+1$ and let
$S$ be a $\gamma _{\times (2k+1),t}(G_{I})$-set. Without loss of
generality, we may assume that for each $1\leq i\leq n-1$, $\mid
S\cap X_{i}\mid =2k+1$ and $\mid S\cap X_{n}\mid =2k+2$. Similar to
the proofs of the previous theorems, we may partition every $S\cap
X_{i}$\ to $k$ $2$-subsets $D_{j}^{(i)}$, where $1\leq j\leq k$,
such that $D_{j}^{(1)}\cup D_{j}^{(2)}\cup ...\cup D_{j}^{(n)}$ is a
union of some disjoint cycles and there is a corresponding partition $%
G=HLD(G_{1}^{(i)},G_{2}^{(i)},...,G_{t_{i}}^{(i)})$ of Hamiltonian
subgraphs for it, and also $\cup _{1\leq i\leq n-1}(S-(\cup _{1\leq
j\leq k}D_{j}^{(i)}))$ makes a blue matching $M_{I}$\ in $G_{I}$ of
size $\lfloor n/2\rfloor $. It can be easily verified that
$M=\{x_{i}x_{j}\mid (x_{i}x_{j},x_{j}x_{i})\in M_{I}\}$ is a maximum
matching in $G$ of size $\lfloor n/2\rfloor $\ such that does not
saturate $x_{n}$ and for every partition
$G=HLD(G_{1}^{(i)},G_{2}^{(i)},...,G_{t_{i}}^{(i)})$ of Hamiltonian
subgraphs it satisfies in the condition (\ref{eqq}), and so $G$ is a
kHLMM-graph.
\end{proof}

\section{$k$-tuple total domination number in the inflation of a\\
connected graph which has a cut-edge or cut-vertex}

In the next theorem we give upper and lower bounds for the $k$-tuple
total domination number of the inflation of a graph $F$ which
contains a cut-edge $e,$ in terms on the $k$-tuple total domination
numbers of the inflation of the components of $F-e$.

\begin{theorem}
\label{LUbounds,cutedge} Let $F$ be a graph with a cut-edge $e$ such
that $G$ and $H$ are the components of $F-e$. If $2\leq k\leq \min
\{\delta (G),\delta (H)\}$, then
\begin{equation*}
\gamma _{\times k,t}(G_{I})+\gamma _{\times k,t}(H_{I})-k\leq \gamma
_{\times k,t}(F_{I})\leq \gamma _{\times k,t}(G_{I})+\gamma _{\times
k,t}(H_{I}).
\end{equation*}
\end{theorem}

\begin{proof}
Let $V(G)=\{x_{i}\mid 1\leq i\leq n\}$, $V(H)=\{y_{i}\mid 1\leq i\leq m\}$.
Without loss of generality, we may assume that $e=x_{1}y_{1}$. Then $%
V(F_{I})=V(G_{I})\cup V(H_{I})\cup \{x_{1}y_{1},y_{1}x_{1}\}$ and

\begin{equation*}
\begin{array}{lll}
E(F_{I}) & = & E(G_{I})\cup E(H_{I})\cup \{(x_{1}x_{j},x_{1}y_{1})\mid
x_{1}x_{j}\in X_{1}\} \\
& \cup & \{(y_{1}y_{j},y_{1}x_{1})\mid y_{1}y_{j}\in Y_{1}\}\cup
\{(x_{1}y_{1},y_{1}x_{1})\}.
\end{array}
\end{equation*}
Also let $X_{1}^{\prime }=X_{1}\cup \{x_{1}y_{1}\}$ and
$Y_{1}^{\prime }=Y_{1}\cup \{y_{1}x_{1}\}$. Let $S_{G}$ and $S_{H}$
be respectively $\gamma _{\times k,t}(G_{I})$-set and $\gamma
_{\times k,t}(H_{I})$-set. Since $S_{G}\cup S_{H}$ is a kTDS of
$F_{I}$ with cardinal $\gamma _{\times k,t}(G_{I})+\gamma _{\times
k,t}(H_{I})$, then $\gamma _{\times k,t}(F_{I})\leq \gamma _{\times
k,t}(G_{I})+\gamma _{\times k,t}(H_{I})$.

Let now $S_{F}$ be a $\gamma _{\times k,t}(F_{I})$-set. If $%
S_{F}\cap \{x_{1}y_{1},y_{1}x_{1}\}=\emptyset $, then $S_{F}\cap
V(G_{I})$ and $S_{F}\cap V(H_{I})$ are respectively $k$-tuple total
dominating sets of $G_{I}$\ and $H_{I}$ and hence

\begin{equation*}
\begin{array}{lll}
\gamma _{\times k,t}(G_{I})+\gamma _{\times k,t}(H_{I}) & \leq & \mid
S_{F}\cap V(G_{I})\mid +\mid S_{F}\cap V(H_{I})\mid \\
& = & \mid S_{F}\mid \\
& = & \gamma _{\times k,t}(F_{I}).
\end{array}
\end{equation*}
Therefore, we assume that $S_{F}\cap \{x_{1}y_{1},y_{1}x_{1}\}\neq
\emptyset $, and in the next two cases we will complete our proof.

\textbf{Case 1.} $\mid S_{F}\cap \{x_{1}y_{1},y_{1}x_{1}\}\mid =1.$

\noindent Let $S_{F}\cap \{x_{1}y_{1},y_{1}x_{1}\}=\{x_{1}y_{1}\}$. Then $%
S_{F}\cap V(H_{I})$ is a kTDS of $H_{I}$ and $\mid S_{F}\cap
X_{1}\mid \geq k$. Since $k\geq 2$ and each clique of every inflated
graph contains at least $k$ vertices of every kTDS and $\mid
S_{F}\cap X_{1}\mid >k$\ follows $\mid S_{F}\cap Y_{1}^{\prime }\mid
=k-1$, then $\mid S_{F}\cap X_{1}\mid =k$. If $deg_{G}(x_{1})=k$,
then $(S_{F}\cap V(G_{I}))\cup \{x_{i}x_{1}\mid x_{1}x_{i}\in
X_{1}\}$\ is a kTDS of $G_{I}$\ with cardinal at most $\mid
S_{F}\cap V(G_{I})\mid +k$ and hence

\begin{equation*}
\begin{array}{lll}
\gamma _{\times k,t}(G_{I})+\gamma _{\times k,t}(H_{I}) & \leq & \mid
S_{F}\cap V(G_{I})\mid +k+\mid S_{F}\cap V(H_{I})\mid \\
& = & \gamma _{\times k,t}(F_{I})+k-1.
\end{array}
\end{equation*}
Otherwise, for every $x_{1}x_{j}\in X_{1}-S_{F}$%
, $(S_{F}\cap V(G_{I}))\cup \{x_{1}x_{j}\}$\ is a kTDS of $G_{I}$
and hence

\begin{equation*}
\begin{array}{lll}
\gamma _{\times k,t}(G_{I})+\gamma _{\times k,t}(H_{I}) & \leq & \mid
(S_{F}\cap V(G_{I}))\cup \{x_{1}x_{j}\}\mid +\mid S_{F}\cap V(H_{I})\mid \\
& = & \gamma _{\times k,t}(F_{I}).
\end{array}
\end{equation*}

\textbf{Case 2.} $\mid S_{F}\cap \{x_{1}y_{1},y_{1}x_{1}\}\mid =2.$

Since $\mid S_{F}\cap X_{1}^{\prime }\mid \geq k$, $\mid S_{F}\cap
Y_{1}^{\prime }\mid \geq k$ and $\{x_{1}y_{1},y_{1}x_{1}\}\subseteq
S_{F}$, then $\mid S_{F}\cap X_{1}\mid =k-1$ or $\mid S_{F}\cap
Y_{1}\mid =k-1 $. Let $\mid S_{F}\cap X_{1}\mid \geq \mid S_{F}\cap
Y_{1}\mid =k-1$. If $deg_{H}(y_{1})=k$, then there exists
$y_{1}y_{j}\in Y_{1}-S_{F}$ such that $(S_{F}\cap V(H_{I}))\cup
\{y_{1}y_{j},y_{j}y_{1}\}$ is a kTDS of $H_{I}$. If
$deg_{H}(y_{1})\geq k+1$, then there are two disjoint vertices
$y_{1}y_{j},y_{1}y_{i}\in Y_{1}-S_{F}$ such that $(S_{F}\cap
V(H_{I}))\cup \{y_{1}y_{j},y_{1}y_{i}\}$ is a kTDS of $H_{I}$.

Now we give a $k$-tuple total dominating set for $G_{I}$ in all
possible cases. If $\mid S_{F}\cap X_{1}^{\prime }\mid \geq k+1$,
then $S_{F}\cap V(G_{I})$ is a kTDS of $G_{I}$. Let $\mid S_{F}\cap
X_{1}\mid =k$ and let $deg_{G}(x_{1})=k$. Then $(S_{F}\cap
V(G_{I}))\cup \{x_{i}x_{1}\mid x_{1}x_{i}\in X_{1}\}$ is a kTDS of
$G_{I}$ with cardinal at most $\mid S_{F}\cap V(G_{I})\mid +k$. If
either $\mid S_{F}\cap X_{1}\mid =k$ and $deg_{G}(x_{1})=k+1$ or
$\mid S_{F}\cap X_{1}\mid =k-1$ and $deg_{G}(x_{1})=k$, then for
each $x_{1}x_{j}\in X_{1}-S_{F}$ the set $(S_{F}\cap V(G_{I}))\cup
\{x_{1}x_{j},x_{j}x_{1}\}$ is a kTDS of $G_{I}$. Finally, if either
$\mid S_{F}\cap X_{1}\mid =k$ and $deg_{G}(x_{1})\geq k+2$ or $\mid
S_{F}\cap X_{1}\mid =k-1$ and $deg_{G}(x_{1})\geq k+1$, then for
every two disjoint vertices $x_{1}x_{j},x_{1}x_{i}\in X_{1}-S_{F}$,
the set $(S_{F}\cap V(G_{I}))\cup \{x_{1}x_{j},x_{1}x_{i}\}$ is a
kTDS of $G_{I}$. Thus in the Case 2 we proved that $\gamma _{\times
k,t}(G_{I})+\gamma _{\times k,t}(H_{I})-k\leq \gamma _{\times
k,t}(F_{I})$.

With comparing the obtained bounds in Cases 1 and 2, we have $\gamma
_{\times k,t}(G_{I})+\gamma _{\times k,t}(H_{I})-k\leq \gamma
_{\times k,t}(F_{I})$, and this completes our proof.
\end{proof}

By closer look at the proof of Theorem \ref{LUbounds,cutedge} we
have the next theorem.

\begin{theorem}
\label{LUbounds,cutedge, strict} Let $F$ be a graph with a cut-edge
$e$ such that $G$ and $H$ are the components of $F-e$. If $2\leq k <
\min \{\delta (G),\delta (H)\}$, then
\begin{equation*}
\gamma _{\times k,t}(G_{I})+\gamma _{\times k,t}(H_{I})-2\leq \gamma
_{\times k,t}(F_{I})\leq \gamma _{\times k,t}(G_{I})+\gamma _{\times
k,t}(H_{I}).
\end{equation*}
\end{theorem}

We now calculate the $k$-tuple total domination number of the
inflation of the complete graphs and then continue our discussion.

\begin{proposition}
\label{ K_n} Let $n>k\geq 2$. Then every complete graph $K_{n}$\ is
$\lfloor (n-1)/2\rfloor $-Hamiltonian-like decomposable graph and
\begin{equation*}
\gamma _{\times k,t}((K_{n})_{I})=\left\{
\begin{array}{cc}
nk+1 & \mbox{if }k\mbox{ and }n\mbox{ are both odd } \\
nk & \mbox{otherwise }
\end{array}
.\right.
\end{equation*}
\end{proposition}

\begin{proof}
Let $V(G)=\{i\mid 1\leq i\leq n\}$. Since for any $1\leq i\leq
\lfloor (n-1)/2\rfloor $ the edge set $E_{i}=\{(j,j+i)\mid 1\leq
j\leq n\}$\ is a union of some disjoint cycles and $\cup _{1\leq
i\leq \lfloor (n-1)/2\rfloor }E_{i}$\ is a partition of $V(K_{n})$,
then $K_{n}$\ is $\lfloor (n-1)/2\rfloor $-Hamiltonian-like
decomposable graph. Since $M=\{(i,i+\lfloor n/2\rfloor )\mid 1\leq
i\leq \lfloor n/2\rfloor \}$\ is respectively a perfect or maximum
matching of size $\lfloor n/2\rfloor $\ of $K_{n}$, when $n$ is
respectively even or odd, then Theorems \ref{gamma=(2k+1)n} and
\ref{gamma=n(2k+1)+1} complete our proof.
\end{proof}

\begin{proposition}
\label{K_n,K_m,cutedge} Let $2\leq k<n\leq m$ and let $F$ be a graph with a cut-edge $e$ such that $%
G=K_{n}$ and $H=K_{m}$ are the components of $F-e$. Then
\begin{equation*}
\gamma _{\times k,t}(F_{I})=\left\{
\begin{array}{cc}
k(n+m)+1 & \mbox{if }k\mbox{ is odd and }n\equiv m+1\mbox{ }(\mbox{mod }2)
\\
k(n+m) & \mbox{otherwise }
\end{array}
.\right.
\end{equation*}
\end{proposition}

\begin{proof}
Let $V(G)=\{x_{i}\mid 1\leq i\leq n\}$, $V(H)=\{y_{i}\mid 1\leq
i\leq m\}$ and $e=x_{n}y_{m}$. Since every complete graph $K_{t}$\
is $\lfloor (t-1)/2\rfloor $-Hamiltonian-like decomposable graph and
$n\leq m$, then $F$ is $\lfloor (n-1)/2\rfloor $-Hamiltonian-like
decomposable graph. We now continue our discussion in the next two
cases.

\textbf{Case 1.} $n\equiv m+1\mbox{ }(\mbox{mod }2)$.

\noindent If $k$ is odd, then Theorem \ref{gamma>=nk+1} follows that
$\gamma _{\times k,t}(F_{I})\geq k(n+m)+1$. Without loss of
generality, we may assume that $n$ is odd and $m$ is even. Then
$\gamma _{\times k,t}(G_{I})=kn+1$ and $\gamma _{\times
k,t}(H_{I})=km$, by Proposition \ref{ K_n}. If $S_{G}$ and $S_{H}$\
are respectively $\gamma _{\times k,t}(G_{I})$-set and $\gamma
_{\times k,t}(H_{I})$-set, then $S_{G}\cup S_{H}$\ is a kTDS of
$F_{I}$ with cardinal $k(n+m)+1$ and so $\gamma _{\times
k,t}(F_{I})=k(n+m)+1$. If $k$ is even, then similarly it can be
verified that $\gamma _{\times k,t}(F_{I})=k(n+m)$.

\textbf{Case 2.} $n\equiv m\mbox{ }(\mbox{mod }2)$.

\noindent Then Theorem \ref{LU.bounds} follows that $\gamma _{\times
k,t}(F_{I})\geq k(n+m)$. If either $n\equiv m\equiv 0\mbox{
}(\mbox{mod }2)$ or $n\equiv m\equiv 1\mbox{ }(\mbox{mod }2)$ and
$k$ is even, then $\gamma _{\times k,t}(G_{I})=kn$, and $\gamma
_{\times k,t}(H_{I})=km$, by Proposition \ref{ K_n}. If $S_{G}$ and
$S_{H}$\ are respectively $\gamma _{\times
k,t}(G_{I}) $-set and $\gamma _{\times k,t}(H_{I})$-set, then obviously $%
S_{G}\cup S_{H}$\ is a kTDS of $F_{I}$ with cardinal $k(n+m)$ and so $%
\gamma _{\times k,t}(F_{I})=k(n+m)$.

Let now $n\equiv m\equiv 1\mbox{ }(\mbox{mod }2)$ and let $k$ be
odd. Then $\gamma _{\times k,t}(G_{I})=kn+1$ and $\gamma _{\times
k,t}(H_{I})=km+1$, by Proposition \ref{ K_n}. Let $S_{G}=S_{1}\cup
\{\alpha ,\beta \}$ be the given $\gamma _{\times k,t}(G_{I})$-set
in the second paragraph of the proof of Theorem
\ref{gamma=n(2k+1)+1} such that $S_{1}=V(M_{I})\cup S^{(1)}\cup
S^{(2)}\cup ...\cup S^{(k)}$ and $\alpha ,\beta \in
X_{n}-(S_{1}-V(M_{I}))$. Similarly, let $S_{H}=S_{1}^{\prime }\cup
\{\alpha ^{\prime },\beta ^{\prime }\}$ be the given $\gamma _{\times k,t}(H_{I})$%
-set in the second paragraph of the proof of Theorem \ref{gamma=n(2k+1)+1} such that $%
S_{1}^{\prime }=V(M_{I}^{\prime })\cup S^{\prime (1)}\cup S^{\prime
(2)}\cup ...\cup S^{\prime (k)}$ and $\alpha ^{\prime },\beta
^{\prime }\in Y_{m}-(S_{1}^{\prime }-V(M_{I}^{\prime }))$. Then
obviously $S=((S_{G}\cup S_{H})-\{\alpha ,\beta ,\alpha ^{\prime
},\beta ^{\prime }\})\cup \{x_{n}y_{m},y_{m}x_{n}\}$ is a kTDS of
$F_{I}$ with cardinal $k(n+m)$ and so $\gamma _{\times
k,t}(F_{I})=k(n+m)$.
\end{proof}

Proposition \ref{ K_n} follows that if $G=K_{n}$ and $H=K_{m}$\ are
complete graphs, then
\begin{equation*}
\gamma _{\times k,t}(G_{I})+\gamma _{\times k,t}(H_{I})=\left\{
\begin{array}{ll}
k(n+m) & \mbox{if }k\mbox{ is odd and }m\mbox{ and }n\mbox{ are both
even,}
\\
k(n+m)+1 & \mbox{if }k\mbox{ is odd and }n\equiv m+1\mbox{ (mod }2), \\
k(n+m)+2 & \mbox{if }k\mbox{, }m\mbox{ and }n\mbox{ are odd.}
\end{array}
\right.
\end{equation*}
Thus Proposition \ref{K_n,K_m,cutedge} follows the next result that
states the given bounds in Theorem \ref{LUbounds,cutedge} are sharp.

\begin{corollary}
\label{K_n,K_m,cutegde,2} Let $2\leq k<n\leq m$ and let $F$ be a graph with a cut-edge $e$ such that $%
G=K_{n}$ and $H=K_{m}$ are the components of $F-e$. Then
\begin{equation*}
\gamma _{\times k,t}(F_{I})=\left\{
\begin{array}{cc}
\gamma _{\times k,t}(G_{I})+\gamma _{\times k,t}(H_{I})-2 & \mbox{if }k%
\mbox{, }m\mbox{ and }n\mbox{ are odd,} \\
\gamma _{\times k,t}(G_{I})+\gamma _{\times k,t}(H_{I}) & \mbox{otherwise }.
\end{array}
\right.
\end{equation*}
\end{corollary}

Now in the next theorem we give upper and lower bounds for the
$k$-tuple total domination number of the inflation of a graph $F,$
which contains a cut-vertex $v$, in terms on the $k$-tuple total
domination numbers of the inflation of the $v$-components of $F-v$.

\begin{theorem}
\label{LUbound.cutver} Let $F$ be a graph with a cut-vertex $v$ such
that $G^{1}$, $G^{2}$, ..., $G^{m}$ are all $v$-components of $F-v$.
If $2\leq k<\min \{\delta (G^{i})\mid 1\leq i\leq m\}$, then
\begin{equation*}
\Sigma _{1\leq i\leq m}\gamma _{\times k,t}(G_{I}^{i})-m(k+1)+k\leq \gamma
_{\times k,t}(F_{I})\leq \Sigma _{1\leq i\leq m}\gamma _{\times
k,t}(G_{I}^{i}),
\end{equation*}
and \label{t:13}the upper bound $\Sigma _{1\leq i\leq m}\gamma
_{\times k,t}(G_{I}^{i})$ is sharp.
\end{theorem}

\begin{proof}
Let $V(G^{i})=\{x_{j}^{i}\mid 1\leq j\leq n(G^{i})\}$. Without loss of
generality, we may suppose that $x_{1}^{1}=...=x_{1}^{m}=v$. Then $%
V(F_{I})=\cup _{1\leq i\leq m}V(G_{I}^{i})$ and

\begin{equation*}
E(F_{I})=(\cup _{1\leq i\leq m}E(G_{I}^{i}))\cup
\{(x_{1}^{i}x_{l}^{i},x_{1}^{j}x_{t}^{j})\mid x_{1}^{i}x_{l}^{i}\in X^{i},%
\mbox{ and }x_{1}^{j}x_{t}^{j}\in X^{j}\mbox{, for }1\leq i<j\leq
m\}.\noindent
\end{equation*}
Let $S^{i}$ be a $\gamma _{\times k,t}(G_{I}^{i})$-set. Since $\cup
_{1\leq i\leq m}S^{i}$ is a kTDS of $F_{I}$ with cardinal $\Sigma
_{1\leq i\leq m}\gamma _{\times k,t}(G_{I}^{i})$, then $\gamma
_{\times k,t}(F_{I})\leq \Sigma _{1\leq i\leq m}\gamma _{\times
k,t}(G_{I}^{i})$.

Let now $S$ be a $\gamma _{\times k,t}(F_{I})$-set. Let $S^{i}=S\cap
V(G_{I}^{i}) $, where $1\leq i\leq m$. Then each $S^{i}$\ is a kTDS
of $G_{I}^{i}-X_{1}^{i}$. Let $\mid S\cap X_{1}^{i}\mid =t_{i}$,
where $1\leq i\leq m$. Then $\Sigma _{1\leq i\leq m}t_{i}\geq k$.
Since $k<\delta (G^{i})$, then by adding at most $k+1-t_{i}$ vertices of $%
X_{1}^{i}$ to $S$, we may obtain a kTDS $S^{\prime }$ of $F_{I}$
such that every $S^{\prime }\cap V(G_{I}^{i})$ is a kTDS of
$G_{I}^{i}$. Then

\begin{equation*}
\begin{array}{lll}
\Sigma _{1\leq i\leq m}\gamma _{\times k,t}(G_{I}^{i}) & \leq & \Sigma
_{1\leq i\leq m}\mid S_{F}^{\prime }\cap V(G_{I}^{i})\mid \\
& \leq & \mid S_{F}\mid +m(k+1)-\Sigma _{1\leq i\leq m}t_{i} \\
& \leq & \gamma _{\times k,t}(F_{I})+m(k+1)-k
\end{array}
\end{equation*}
and hence
\begin{equation*}
\Sigma _{1\leq i\leq m}\gamma _{\times k,t}(G_{I}^{i})-m(k+1)+k\leq \gamma
_{\times k,t}(F_{I})\leq \Sigma _{1\leq i\leq m}\gamma _{\times
k,t}(G_{I}^{i}).
\end{equation*}

Now we show that the upper bound $\Sigma _{1\leq i\leq m}\gamma
_{\times k,t}(G_{I}^{i})$ is sharp. Let $F$ be a graph with a
cut-vertex $v$ such that $G^{1}$, $G^{2}$, ..., $G^{m}$ are all
$v$-components of $F-v$ and $\gamma _{\times
k,t}(G_{I}^{i})=n(G^{i})k$.
Consider $V(G^{i})=\{x_{j}^{i}%
\mid 1\leq j\leq n(G^{i})\}$ and $x_{1}^{1}=...=x_{1}^{m}=v$. Let $%
Y_{v}^{F} $ be the respective red clique with vertex $v$ in $F$. Let $%
S^{i}$ be a $\gamma _{\times k,t}(G_{I}^{i})$-set, where $1\leq i\leq m$%
. Then every clique in $G_{I}^{i}$ contains exactly $k$ vertices of
$S^{i}$. Then $S=\cup _{1\leq i\leq m}S^{i}$ is a kTDS of $F_{I}$
with cardinal $\Sigma _{1\leq i\leq m}\gamma _{\times
k,t}(G_{I}^{i})=\Sigma _{1\leq i\leq m}n(G^{i})k$ such that
$Y_{v}^{F}$ contains $mk$ vertices of $S $.

We claim that $S$ has minimum cardinal among of all $k$-tuple total
dominating sets of $F_I$. Observation \ref{obser} follows that every
red clique other than $Y_{v}^{F}$ must contain at least $k$ vertices
of every kTDS of $F_{I}$. Thus we can not reduce the number of
vertices of $S$ in cliques except probably $Y_{v}^{F}$. Since also
reducing the number of the vertices of $S\cap Y_{v}^{F}$ reduce the
cardinal of $k$-tuple total domination number of $G_{I}^{i}$, then
we can not reduce it, by Observation \ref{obser}. Therefore $S$ is a
minimal kTDS of $F_{I}$. Now let $S^{\prime }$ be an arbitrary
$\gamma _{\times k,t}(F_{I})$-set with cardinal less than $\Sigma
_{1\leq i\leq m}n(G^{i})k$. Then, by the previous discussion, there
exists a $v$-component $G^{i}$ of $F-v $ and a clique $X$ of it
other than $X_{1}^{i}=Y_{v}^{F}\cap V(G_{I}^{i})$ such that $\mid
S^{\prime }\cap X\mid <k$. But this is not possible, by Observation
\ref{obser}. Therefore $S$ is a $\gamma _{\times k,t}(F_{I})$-set
and so $\gamma _{\times k,t}(F_{I})=\Sigma _{1\leq i\leq m}\gamma
_{\times k,t}(G_{I}^{i})=\Sigma _{1\leq i\leq m}n(G^{i})k.$
\end{proof}

Let $G^{1}$, $G^{2}$, ..., $G^{m}$ and $F$ be the given graphs in
the second part of the proof of Theorem \ref{LUbound.cutver}. Then
we see that $n(F)=\Sigma _{1\leq i\leq m}n(G^{i})-m+1$ and
\begin{equation*}
\begin{array}{lll}
\gamma _{\times k,t}(F_{I}) & = & \Sigma _{1\leq i\leq m}n(G^{i})k \\
& = & n(F)k+(m-1)k \\
& \leq & n(F)(k+1)-1.
\end{array}
\end{equation*}
Thus this family of graphs are examples of the graphs $G$ of order
$n$, which $\gamma _{\times k,t}(G_{I})=nk+\alpha k\leq n(k+1)-1$,
where $\alpha $ is an arbitrary positive integer.


\section{$k$-tuple total domination number in the inflation of some graphs}

In section 3, we calculated the $k$-tuple total domination number of
the inflation of the complete graphs. Now we find this number in the
inflation of the generalized Petersen graphs, Harary graphs and
complete bipartite graphs. Also we give an upper bound for this
number when our graph is a complete multipartite graph.

In \cite{Wat69}, Watkins introduced the notion of generalized
Petersen graph (GPG for short) as follows: for any integer $n\geq 3$
let $Z_{n}$\ be additive group on $\{1,2,...,n\}$ and $m\in
Z_{n}-\{0\}$, the graph $P(n,m)$ is defined on the set
$\{a_{i},b_{i}\mid i\in Z_{n}\}$ of $2n$ vertices with edges
$a_{i}a_{i+1}$, $a_{i}b_{i}$, $b_{i}b_{i+m}$ for all $i$. If
$m=n/2$, then every vertex $b_{i}$\ has degree $2$ and every
vertex $a_{i}$\ has degree $3$, and otherwise $P(n,m)$\ is 3-regular. Thus $%
\gamma _{\times 3,t}((P(n,m)_{I})=n(G_{I})=6n$, where $m\neq n/2$.
Since $M=\{a_{i}b_{i}\mid i\in Z_{n}\}$ is a perfect matching in
$P(n,m)$, then $S=$\ $\{a_{i}b_{i},b_{i}a_{i}\mid i\in Z_{n}\}$\ is
a $\gamma _{t}((P(n,m))_{I})$-set and so $\gamma
_{t}((P(n,m)_{I})=2n$. In the next proposition we calculate $\gamma
_{\times 2,t}((P(n,m)_{I})$.

\begin{proposition}
\label{2,P(n,m)} Let $n\geq 3$ and $m\geq 1$ be integers. Then
\begin{equation*}
\gamma _{\times 2,t}((P(n,m))_{I})=\left\{
\begin{array}{cc}
4n+2 & \mbox{if }m=n/2\mbox{ is odd } \\
4n & \mbox{otherwise }
\end{array}
.\right.
\end{equation*}
\end{proposition}

\begin{proof}
Let $G=P(n,m)$. We first assume that $m\neq n/2$ and $d$ is the
greatest common divisor of $m$ and $n$. Then the induced subgraph by $\{b_{i}\mid i\in Z_{n}\}$ of $%
G $ has a partition to $d$ disjoint cycle or cycles $%
C_{i}:b_{i}b_{i+m}b_{i+2m}...b_{i+\alpha -m}$, where $1\leq i\leq d$\ and $%
\alpha =\min \{tm\mid tm\equiv 0\mbox{ mod }n\}$. Since the induced
subgraph by $\{a_{i}\mid i\in Z_{n}\}$ of $G$ is cycle $%
C_{a}:a_{1}a_{2}a_{3}...a_{n}$, then $G$\ is a Hamiltonian-like
decomposable graph and Theorem \ref {gamma=2kn} follows $\gamma
_{\times 2,t}(G_{I})=4n$.

Let now $m=n/2$. Then $b_{i}b_{j}\in E(G)$ if and only if $j\equiv
i+m\mbox{ (mod }n)$. Then every vertex $b_{i}$\ has degree $2$ and
every vertex $a_{i}$ has degree $3$. Then there exist $\lfloor m/2
\rfloor $ disjoint cycles
$b_{i}a_{i}a_{i+1}b_{i+1}b_{i+1+m}a_{i+1+m}a_{i+m}b_{i+m}$ of length
$8$. If $m$ is even, then these cycles are a partition of $V(G)$.
Hence $G$ is a Hamiltonian-like decomposable graph and Theorem \ref
{gamma=2kn} follows $\gamma _{\times 2,t}(G_{I})=4n$. Otherwise
these cycles are a partition of $V(G)-\{a_{m},b_{m},b_{n},a_{n}\}$.
We notice that the
induced subgraph of $G$ by $\{a_{m},b_{m},b_{n},a_{n}\}$ is the path $%
P_{4}:a_{m}b_{m}b_{n}a_{n}$. Set
\begin{equation*}
\begin{array}{lll}
S & = & S_{1}\cup S_{2}\cup ...\cup S_{m^{\prime }} \\
& \cup &
\{a_{m}a_{m+1},a_{m}a_{m-1},a_{m}a_{m};b_{m}a_{m},b_{m}b_{n};b_{n}b_{m},b_{n}a_{n};a_{n}b_{n},a_{n}a_{1},a_{n}a_{n-1}\},
\end{array}
\end{equation*}
where $1\leq i\leq m^{\prime }$ and
\begin{equation*}
\begin{array}{lll}
S_{i} & = &
\{b_{i}b_{i+m},b_{i}a_{i};a_{i}b_{i},a_{i}a_{i+1};a_{i+1}a_{i},a_{i+1}b_{i+1};b_{i+1}a_{i+1},\}
\\
& \cup &
\{b_{i+1}b_{i+1+m};b_{i+1+m}b_{i+1},b_{i+m+1}a_{i+m+1};a_{i+m+1}b_{i+m+1}\}
\\
& \cup &
\{a_{i+m+1}a_{i+m};a_{i+m}a_{i+m+1},a_{i+m}b_{i+m};b_{i+m}a_{i+m},b_{i+m}b_{i}\}.
\end{array}
\end{equation*}
One can verify that $S$ is a minimum DTDS of $G_{I}$ and so $\gamma
_{\times 2,t}(G_{I})=4n+2$.
\end{proof}

We now consider Harary graphs which make a great family of graphs. Given $%
m<n $, place \ $n$ vertices $1$, $2$, $...$, $n$ around a circle, equally
spaced. If $m$ is even, form $H_{m,n}$ by making each vertex adjacent to the
nearest $m/2$ vertices in each direction around the circle. If $m$ is odd
and $n$ is even, form $H_{m,n}$\ by making each vertex adjacent to the
nearest $(m-1)/2$ vertices in each direction and to the diametrically
opposite vertex. In each case, $H_{m,n}$ is \ $m$-regular. When $m$ and $n$
are both odd, index the vertices by the integers modulo $n$. Construct $%
H_{m,n}$ from $H_{m-1,n}$ by adding the edges $(i,i+(n-1)/2),$ for
$0\leq i\leq (n-1)/2$ (see \cite {West}).

\begin{proposition}
\label{Harary} Let $2\leq k\leq m<n$ be integers. Then the Harary graph $H_{m,n}$\ is $%
\lfloor m/2\rfloor $-Hamiltonian-like decomposable graph and
\begin{equation*}
\gamma _{\times k,t}((H_{m,n})_{I})=\left\{
\begin{array}{cc}
nk+1 & \mbox{if }k\mbox{ and }n\mbox{ are both odd } \\
nk & \mbox{otherwise }
\end{array}
.\right.
\end{equation*}
\end{proposition}

\begin{proof}
Since for each $1\leq i\leq m$\ the edge subset $E_{i}=\{(j,j+i)\mid
1\leq j\leq n\}$\ is a union of some disjoint cycles and $\cup
_{1\leq i\leq m}E_{i}$ is a partition of $V(H_{m,n})$, then $H_{m,n}$ is a $m$%
-Hamiltonian-like decomposable graph. Let $m$ be odd. If $n$ is even
or odd, then $M=\{(i,i+\lfloor n/2\rfloor )\mid 1\leq i\leq \lfloor
n/2\rfloor \}$ is respectively a perfect or maximum matching of size
$\lfloor n/2\rfloor $ of $H_{m,n}$. Then Theorems
\ref{gamma=(2k+1)n} and \ref{gamma=n(2k+1)+1} complete our proof.
\end{proof}

In the following two theorems we consider the complete bipartite
graphs $K_{p,q}$. First let $p=q$.

\begin{proposition}
\label{K_p,p} For integers $p\geq k\geq 2$, let $G$ be the complete
bipartite graph $K_{p,p}$. Then\ $G$ is a $(\lfloor p/2\rfloor
-1)$HLPM-graph if $p$ is even, otherwise is a $\lfloor p/2\rfloor
$-Hamiltonian-like decomposable graph and so $\gamma _{\times
k,t}(G_{I})=2pk$.
\end{proposition}

\begin{proof}
We consider the partition $X\cup Y$ for $V(G)$, where $X=\{x_{i}\mid
1\leq i\leq p\}$ and $Y=\{y_{i}\mid 1\leq i\leq p\}$. For $0\leq
j\leq \lfloor p/2\rfloor -1 $, we choose $\lfloor p/2\rfloor $\
sequences on $X\cup Y$\ of length $2p$ that are alternatively from
$X$ and $Y$ with
starting of vertex $x_{1}$ such that every three consequence numbers of them are $x_{i}$, $y_{i+j}$, and $%
x_{i+(2j+1)}$. Let $0\leq j\leq \lfloor p/2\rfloor -2$.\ If $p$ does
not divided by $2j+1$, then $j$-th sequence makes the cycle

\begin{equation*}
C_{j}:x_{1}y_{j+1}x_{2j+2}y_{3j+2}...x_{p-2j}y_{p-j}
\end{equation*}
but if $p=(2j+1)t$, for some positive integer $t$, then it makes
$2j+1$ disjoint cycles
\begin{equation*}
C_{i}^{j}:x_{i}y_{i+j}x_{i+(2j+1)}y_{i+(3j+1)}...x_{i+(t-1)(2j+1)}y_{i+(t-1)(2j+1)+j}
\end{equation*}
of length $t$, where $1\leq i\leq 2j+1$. We notice that for odd $p$
and $j=\lfloor p/2\rfloor -1$ there exists another cycle of length
$2p$ that is disjoint of the other cycles. When $p$ is even and
$j=\lfloor p/2\rfloor -1$, the corresponding sequence makes a
perfect matching $M$ that is disjoint of the other cycles. Then
Theorems \ref{gamma=2kn} and \ref{gamma=(2k+1)n} follow $\gamma
_{\times k,t}(G_{I})=2pk.$
\end{proof}

\begin{proposition}
\label{K_p,q} For integers $q\geq p>k\geq 2$, let $G$ be the complete bipartite graph $K_{p,q}$%
. Then\ $\gamma _{\times k,t}(G_{I})=2pk+(q-p)(k+1)$.
\end{proposition}

\begin{proof}
Let $S$ be an arbitrary $\gamma _{\times k,t}(G_{I})$-set such that
$\alpha $ red cliques of $G_{I}$ contain $k$ vertices and other $%
p+q-\alpha $ red cliques of $G_{I}$ contain $k+1$ vertices of $S$.
Since $G$ is bipartite, then $\alpha /2$ cliques must be selected
among of the $q$ red cliques $Y_{i}$, where $1\leq i\leq q$, and the
other second $\alpha /2$ cliques must be selected among of the $p$
red cliques $X_{i}$, where $1\leq i\leq p$. We notice that this
choosing is possible. Because, by Proposition \ref{K_p,p},
$K_{p,p}$ is respectively $(\lfloor p/2\rfloor -1)$HLPM-graph and $\lfloor p/2\rfloor $%
-Hamiltonian-like decomposable graph, when $p$ is respectively even
or odd. Thus $\alpha \leq 2p$ and so
\begin{equation*}
\begin{array}{lll}
\gamma _{\times k,t}(G_{I}) & = & \min \{\mid S\mid :S\mbox{ is a }%
\mbox{kTDS of }G_{I}\} \\
& = & \min \{\alpha k+(q+p-\alpha )(k+1):0\leq \alpha \leq 2p\} \\
& = & \min \{(q+p)(k+1)-\alpha :0\leq \alpha \leq 2p\} \\
& = & (q+p)(k+1)-2p \\
& = & 2pk+(q-p)(k+1).
\end{array}
\end{equation*}
\end{proof}

We notice that $\gamma _{\times p,t}((K_{p,q})_{I})=2pq$ and for $k=n$, $%
2pq=2pk+(q-p)(k+1)$ if and only if $p=q$. By Theorem
\ref{LU.bounds}, if $k\geq 2$ is integer and $G$ is a graph of order
$n$ with $\delta \geq k$, then $n(k+1)-n\leq \gamma _{\times
k,t}(G_{I})\leq n(k+1)-1$. Therefore Proposition \ref{K_p,q} follows
the next theorem.

\begin{theorem}
\label{exist} For each integers $n$, $k$ and $\ell$ with the
condition
$2\leq k<\ell \leq \lfloor n/2\rfloor $, there exists a graph $G$ of order $n$ such that $%
\gamma _{\times k,t}(G_{I})=n(k+1)-2\ell$.
\end{theorem}

\begin{proof}
Let $G=K_{\ell,n-\ell}$. Then Proposition \ref{K_p,q} follows
\begin{equation*}
\begin{array}{lll}
\gamma _{\times k,t}(G_{I}) & = & 2\ell k+(n-2\ell)(k+1) \\
& = & n(k+1)-2\ell.
\end{array}
\end{equation*}
\end{proof}

The next theorem gives an upper bound for the $k$-tuple total
domination number of the complete multipartite graphs.

\begin{proposition}
\label{com.multipart.} Let $G$ be the complete multipartite graph $K_{n_{1},n_{2},...,n_{m}}$. Let $%
n=n_{1}+...+n_{m}$ and $n^{\prime }=\max \{\sum_{i\in J}n_{i}\mid
J\subseteq \{1,2,..m\}\mbox{ and }\sum_{i\in J}n_{i}\leq n/2\}$.
Then for every $2\leq k<n^{\prime }$,
\[
\gamma _{\times k,t}(G_{I})\leq n(k+1)-2n^{\prime }.
\]
\end{proposition}

\begin{proof}
We assume that $V(G)=X^{(1)}\cup X^{(2)}\cup ...\cup X^{(m)}$ is the
partition of vertices of the graph, where $X^{(i)}=\{x_{j}^{(i)}\mid
1\leq j\leq n_{i}\}$. Let $n^{\prime }=\sum_{i\in J}n_{i}$, for some
$J\subseteq \{1,2,..m\}$. Let $X=\cup _{i\in J}X^{(i)}$ and $Y=\cup
_{i\notin J}X^{(i)}$. Then every vertex of $X$ is adjacent to every
vertex of $Y$. If $H$ is the complete bipartite with the vertex set
$X\cup Y$, then it is a subgraph of $G$ and so $\gamma _{\times
k,t}(G_{I})\leq \gamma _{\times k,t}(H_{I})=n(k+1)-2n^{\prime }$, by
Proposition \ref{K_p,q}.
\end{proof}

In the end of our paper we expose some problems.

\textbf{Problems:}

1. Can be improved the upper bound $n(k+1)-1$ in Theorem
\ref{LU.bounds}?

2. Whether the lower bound $\Sigma _{1\leq i\leq m}\gamma _{\times
k,t}(G_{I}^{i})-m(k+1)+k$\ in Theorem \ref{LUbound.cutver} is sharp?

3. Characterize all graphs $G$ such that $\gamma _{\times
k,t}(G_{I})=nk+1.$

\end{document}